\documentclass[12pt,reqno]{amsart}
\usepackage{amssymb}
\usepackage{graphicx}
\usepackage{bm}
\usepackage{enumerate}
\allowdisplaybreaks[1]

\newcommand{\C}{\mathbb C}
\newcommand{\D}{\mathbb D}

\newcommand{\N}{\mathbb N}
\newcommand{\R}{\mathbb R}
\newcommand{\sS}{\mathbb S}
\renewcommand{\H}{\mathbb H}
\renewcommand{\P}{\mathbb P}

\newcommand{\Z}{\mathbb Z}

\newcommand{\cF}{\mathcal F}

\newcommand{\re}{\mathrm{Re}}
\newcommand{\im}{\mathrm{Im}}

\newcommand{\g}{\gamma}

\newcommand{\h}{\vartheta}

\newcommand{\G}{\mathit{\Gamma}}

\renewcommand{\i}{\mathbf{i}}

\newcommand{\la}{\langle}
\newcommand{\ra}{\rangle}
\newcommand{\tr}{\;^t}

\newcommand{\pr}{\mathrm{pr}}

\newcommand{\comment}[1]{}

\newcommand{\ds}[1]{\displaystyle{#1}}

\title[An analogy of Jacobi's formula]
{An analogy of Jacobi's formula and its applications}
 \author{Chiba Jun}
 \address[Chiba]{NEC Corporation, 
7-1, Shiba 5-chome Minato-ku, Tokyo 108-8001, Japan}
\email{jun.chiba323@nec.com}

\author{Matsumoto Keiji}
\address[Matsumoto]
{Department of Mathematics,
Faculty of Science,
Hokkaido University,
Sapporo 060-0810, Japan
}
 \email{matsu@math.sci.hokudai.ac.jp}

\date{\today}

\theoremstyle{plain} 
\newtheorem{theorem}{\indent\sc Theorem}[section]
\newtheorem{lemma}[theorem]{\indent\sc Lemma}
\newtheorem{cor}[theorem]{\indent\sc Corollary}
\newtheorem{proposition}[theorem]{\indent\sc Proposition}

\theoremstyle{definition} 
\newtheorem{definition}[theorem]{\indent\sc Definition}
\newtheorem{fact}[theorem]{\indent\sc Fact}
\newtheorem{remark}[theorem]{\indent\sc Remark}
\newtheorem{example}[theorem]{\indent\sc Example}

\numberwithin{equation}{section}

\newcommand{\itref}[1]{$(\ref{#1})$}

\begin{document}
\maketitle

\begin{abstract}
We give an analogy of Jacobi's formula, which relates the hypergeometric 
function with parameters $(1/4,1/4,1)$ and theta constants. 
By using this analogy and twice formulas of theta constants, 
we obtain a transformation formula for this hypergeometric function. 
As its application, 
we express the limit of a pair of sequences defined by a mean iteration 
by this hypergeometric function. 
\end{abstract}

\section{Introduction}
Jacobi's formula in elliptic function theory is an equality 
$$F(\frac{1}{2},\frac{1}{2},1;\lambda(\tau))=\h_{00}(\tau)^2,
\quad \lambda(\tau)=\frac{\h_{10}(\tau(z))^4}{\h_{00}(\tau(z))^4},
$$
as holomorphic functions on the upper half plane 
$\H=\{\tau\in \C\mid \im(\tau)>0\}$,   
where $F(a,b,c;z)$ and $\h_{pq}(\tau)$ denote 
the hypergeometric series on $\{z\in \C\mid |z|<1\}$ 
and the theta constant with characteristics $p,q\in \{0,1\}$ on $\H$, 
respectively;  refer to \eqref{def:HGS} and \eqref{eq:theta} 
for their definitions. 
This formula is applied to the study of the arithmetic-geometric mean 
as shown in \cite{BB}, and generalized to Thomae's formula in \cite{Th}. 
In the present, there are several kinds of its analogies, 
which are applied to studies of mean iterations, 
refer to \cite{BB}, \cite{MS}, \cite{MT}, \cite{Sh} and the references therein.

In this paper, we give an analogy of Jacobi's formula, which is an equality 
\begin{equation*}
\label{eq:exJacobi}
F(\frac{1}{4},\frac{1}{4},1;\zeta(\tau))=\h_{00}(\tau)^2,
\quad 
\zeta(\tau)=\frac{4\h_{01}(\tau)^4\h_{10}(\tau)^4}{\h_{00}(\tau)^8},
\end{equation*}
as holomorphic functions on $\H$.
Strictly, we initially give this equality on a fundamental region $\D_{12}$ 
of 
$$\Gamma_{12}=\{g=g_{ij}\in \mathrm{SL}_2(\Z)\mid 
g_{11}g_{12},g_{21}g_{22}\equiv 0 \bmod 2\}$$ 
acting on $\H$, and extend it to 
the whole space $\H$. 
Though this equality is naturally extended to $\H$ 
by the simply connectedness of $\H$,  
$F(\frac{1}{4},\frac{1}{4},1;z)$ is not regarded as single valued 
when $\tau$ runs over $\H$. To understand its behavior well, 
we investigate the monodromy representation of 
the hypergeometric differential equation $\cF(\frac{1}{4},\frac{1}{4},1)$. 
It is realized in not $\mathrm{SL}_2(\Z)$ but $\mathrm{GL}_2(\Z[\i])$, 
and its  projectivization becomes 
$\mathrm{P}\Gamma_{12}=\Gamma_{12}/\{\pm I_2\}$,    
where $\i=\sqrt{-1}$ and $I_2$ is the unit matrix of size $2$.
By using circuit matrices in $\mathrm{GL}_2(\Z[\i])$, we can simplify 
not only the formula for $\h_{00}(\tau)$ with respect to the action 
$\mathrm{P}\Gamma_{12}$ in Fact \ref{fact:trans-theta} but also  
a description of the behavior of $F(\frac{1}{4},\frac{1}{4},1;\tau(z))$; 
for details refer to Lemma \ref{lem:factorG2} 
and Corollary \ref{cor:extend-J-1/4}. 
We also remark that 
the monodromy representation of $\cF(\frac{1}{2},\frac{1}{2},1)$ 
is realized as  
$$\Gamma(2,4)=\{g=(g_{ij})\in \mathrm{SL}_2(\Z)\mid 
g_{11},\g_{22}\equiv 1 \bmod 4,  g_{12},g_{21}\equiv 0\bmod 2\},$$  
which is a subgroup in 
$\Gamma(2)=\{g\in \mathrm{SL}_2(\Z)\mid g\equiv I_2\bmod 2\}$ 
of index $2$, 
and that there is a similar advantage 
in using not $\mathrm{P}\Gamma(2)=\Gamma(2)/\{\pm I_2\}$ 
but $\Gamma(2,4)$ for Jacobi's formula, 
see Lemma \ref{lem:factorG2} and Corollary \ref{cor:Jacobi-formula}.

By using our analogy of Jacobi's formula together with twice formulas  
for theta constants in \eqref{eq:2tau},  
we obtain a transformation formula 
for $F(\frac{1}{4},\frac{1}{4},1;z)$ in Theorem \ref{th:trans}. 
As studied in \cite{Go}, \cite{HKM}, \cite{KS1}, \cite{KS2} and \cite{Ma}, 
some transformation formulas for hypergeometric functions 
are applied to expressing limits of mean iterations. 
By following the idea in \cite{HKM}, we define 
two functions on 
the set 
$$\sS_{(-1,\infty)}=\{(x,rx)\in \R^2\mid x>0,\ -1<r<\infty\}$$ by
$$\mu_1(x,y)=m_A(x,\mu_0(x,y)), \quad \mu_2=\nu(x,\mu_0(x,y)),
$$
where $m_A$, $m_G$ and  $m_H$ are the arithmetic, geometric and 
harmonic means, and 
$$\mu_0=m_G(x,m_A(x,y)),\quad \nu(x,y)=2m_H(x,y)-m_A(x,y).$$
We can easily see  that $\mu_1(x,y)$ is a mean on $\sS_{(-1,\infty)}$, 
however $\mu_2(x,y)$ is not, since the inequality
$\min(x,y)\le \mu_2(x,y)$
does not hold for any $(x,y)\in \sS_{(-1,\infty)}$. 
In spite of this situation, 
we can define a pair of sequences $\{x_n\}$ and $\{y_n\}$  
by the recurrence relation 
$$(x_{n+1},y_{n+1})=(\mu_1(x_n,y_n),\mu_2(x_n,y_n))$$ 
for any given initial term $(x_0,y_0)\in \sS_{(-1,\infty)}$,  
and show in Lemma \ref{lem:limit} that they converge as $n\to \infty$ and 
$\lim\limits_{n\to\infty}x_n=\lim\limits_{n\to\infty}y_n>0$.  
We express this limit by $F(\frac{1}{4},\frac{1}{4},1;z)$ 
in Theorem \ref{th:lim-HGS-rep}.

\section{Fundamental properties of $F(a,b,c;z)$ }
We begin with defining the hypergeometric series.
\begin{definition}
\label{def:HGS}
The hypergeometric series $F(a,b,c;z)$ is defined by 
\begin{equation}
\label{eq:HGS}
F(a,b,c;z)=\sum_{n=0}^\infty \frac{(a,n)(b,n)}{(c,n)(1,n)}z^n, 
\end{equation}
where $z$ is a complex variable,
$a,b,c$ are complex parameters with $c\ne 0,-1,-2,\cdots$, and 
$(a,n)=\G(a+n)/\G(a)=a(a+1)\cdots(a+n-1)$.
It converges absolutely and uniformly on any compact set 
in the unit disk $\{z\in \C\mid |z|<1\}$ for any fixed $a,b,c$.
\end{definition}

This series satisfies the hypergeometric differential equation 
$$\mathcal{F}(a,b,c):\big[z(1-z)\frac{d^2}{dz^2}+\{c-(a+b+1)z\}
\frac{d}{dz}-ab\big]\cdot f(z)=0,$$ 
which is a second order linear ordinary differential equation with 
regular singular points $z=0,1,\infty$. 
The space $\mathcal{S}(a,b,c;\dot z)$ of 
local solutions to $\mathcal{F}(a,b,c)$ around a point 
$\dot z\in Z=\C-\{0,1\}$ is a $2$-dimensional complex vector space, 
its element admits the analytic continuation along any path in $Z$.
In particular, a loop $\rho$ in $Z$ with a base point $\dot z$ leads to 
a linear transformation of $\mathcal{S}(a,b,c;\dot z)$. Thus,  we have 
a homomorphism from the fundamental group $\pi_1(Z,\dot z)$ to 
the general linear group $\mathrm{GL}(\mathcal{S}(a,b,c;\dot z))$ of 
$\mathcal{S}(a,b,c;\dot z)$, which is called the monodromy representation of 
$\cF(a,b,c)$.
We take $\dot z=\frac{1}{2}$ as a base point of $Z$, and 
loops $\rho_0$ and $\rho_1$ with terminal $\dot z$ as 
positively oriented circles with radius $\frac{1}{2}$ and 
center $0$ and $1$, respectively. 
Since $\pi_1(Z,\dot z)$ is a free group generated by $\rho_0$ and $\rho_1$, 
the monodromy representation of $\cF(a,b,c)$ is uniquely characterized 
by the images of $\rho_0$ and $\rho_1$.
We give their explicit forms 
with respect to a basis of $\mathcal{S}(a,b,c;\dot z)$ given by 
Euler type integrals.
\begin{fact}[{\cite[\S 17]{Ki}}]
\begin{enumerate}
\item 
If $\re(b),\re(c-b),\re(1-a)>0$ then the integrals 
$$f_1(z)=\int_1^{1/z} t^{b-1}(1-t)^{c-b-1}(1-tz)^{-a}dt,
\quad 
f_2(z)=\int_0^1 t^{b-1}(1-t)^{c-b-1}(1-tz)^{-a}dt
$$
converge and span the space $\mathcal{S}(a,b,c;\dot z)$. 
Here we assign a branch of the integrand  
$t^{b-1}(1-t)^{c-b-1}(1-tz)^{-a}$ on the open interval $(0,1)$ in the $t$-space 
by $\arg(t)=\arg(1-t)=\arg(1-tz)=0$ for real $z$ near to $\dot z=\frac{1}{2}$, 
and that on the open interval $(1,\frac{1}{z})$ by 
its analytic continuation via the lower half plane of the $t$-space, 
which transforms $\arg(1-t)$ from $0$ to $\pi$.  
Under some generic conditions on $a,b,c$,  
they admit expressions in terms of the hypergeometric series
\begin{align*}
f_1(z)&=e^{\pi\i(c-b-1)}B(c-b,1-a)(1-z)^{c-a-b}F(c-a,c-b,c-a-b+1;1-z),\\ 
&=\frac{e^{2\pi\i(c-b)}\!-\!e^{2\pi\i a}}{e^{2\pi\i a}\!-\!1}
\left(B(b,c\!-\!b)F(a,b,c;z)\!-\!B(b,c\!-\! a\!-\! b)
F(a,b,a\!+\! b\!-\! c\!+\!1;1\!-\! z)\right),\\
f_2(z)&=B(b,c-b)F(a,b,c;z),
\end{align*}
where $B$ denotes the beta function. 
\item 
The loops $\rho_0$ and $\rho_1$  
lead to linear transformations sending the 
basis $\mathbf{F}(z)=\tr(f_1(z),f_2(z))$ of $\mathcal{S}(a,b,c;\dot z)$ 
to $M_0 \mathbf{F}(z)$ and $M_1\mathbf{F}(z)$, where 
\begin{equation}
\label{eq:Cir-Mat}
M_0=
\begin{pmatrix}
e^{-2\pi\i c} & e^{-2\pi\i b}-1  \\  
0 & 1
\end{pmatrix},
\quad 
M_1=
\begin{pmatrix}
e^{2\pi\i(c-a-b)}   & 0\\
1-e^{-2\pi\i a} & 1
\end{pmatrix}.
\end{equation}
Here we exchange the role of $f_1(z)$ and $f_2(z)$ in \cite[\S 17]{Ki}, 
our matrices $M_0$ and $M_1$ are the conjugates of $A_0$ and $A_1$ 
in \cite[p.123]{Ki} by $U=\begin{pmatrix} 0 & 1 \\ 1 & 0\end{pmatrix}$,
respectively.
\end{enumerate}
\end{fact}


\begin{remark}
\label{rem:discontinuous}
We can make the analytic continuation of the hypergeometric series $F(a,b,c;z)$ 
to the simply connected domain $\C-[1,\infty)$ as a solution to $\cF(a,b,c)$.
We use the same symbol $F(a,b,c;z)$ for this continuation, 
which is a single-valued holomorphic function on $\C-[1,\infty)$. 
Moreover, by assigning $F(a,b,c;t)$ on $t\in (1,\infty)$ 
to the limit of $F(a,b,c;z)$ as $z\to t$ 
with $\im(z)>0$, we extend $F(a,b,c;z)$ to a single-valued function 
on $\C-\{1\}$, which is discontinuous along $(1,\infty)$ for general $a,b,c$.
For a path $\rho$ in $Z=\C-\{0,1\}$ ending at $w$,  
$F_\rho(a,b,c;w)$ denotes the continuation of $F(a,b,c;z)$ along $\rho$.
\end{remark}

Though we remove the point $z=1$ from $\C$ for $F(a,b,c;z)$ in 
Remark \ref{rem:discontinuous}, 
there is a formula for the limit of $F(a,b,c;z)$ as 
$z\to 1$ with $z$ in the unit disk.

\begin{fact}[Gauss-Kummer's identity {\cite[Theorem 3.3]{Ki}}]
If 
$\re(c-a-b)>0$ 
then 
\begin{equation}
\label{eq:G-K}
\lim_{z\to 1,|z|<1}F(a,b,c;z)=\frac{\G(c)\G(c-a-b)}{\G(c-a)\G(c-b)}.
\end{equation}
\end{fact}

For the basis $\mathbf{F}(z)=\tr(f_1(z),f_2(z))$ of 
$\mathcal{S}(a,b,c;\dot z)$,  
we have a map from a neighborhood of $\dot z$ to the complex projective
line $\P^1$ by the ratio $f_1(z)/f_2(z)$ of $f_1(z)$ and $f_2(z)$.  
Schwarz's map is define by its analytic continuation to $\C-\{0,1\}$, 
which is multi valued in general. 

\begin{fact}
\label{fact:Schwarz}
If the parameters $a,b,c$ are real and satisfy
$$\frac{1}{|1-c|},\frac{1}{|c-a-b|},\frac{1}{|a-b|}\in 
\{2,3,4\dots\}\cup\{\infty\},\quad 
|1-c|+|c-a-b|+|a-b|<1,$$
then the image of Schwarz's map is isomorphic to an open dense subset of 
the upper half plane $\H=\{\tau\in \C\mid \im(\tau)>0\}$ and its 
inverse is single valued. The projectivization of the image of its monodromy 
representation is conjugate to a discrete subgroup of 
$\mathrm{PSL}_2(\R)=\mathrm{SL}_2(\R)/\{\pm I_2\}$ generated 
by two elements $g_0$ and $g_1$ with 
$$\mathrm{ord}(g_0)=\frac{1}{|1-c|}, \quad 
\mathrm{ord}(g_1)=\frac{1}{|c-a-b|}, \quad 
\mathrm{ord}(g_0^{-1}g_1^{-1})=\frac{1}{|a-b|},
$$
where $I_2$ is the unit matrix of size $2$, and    
$\mathrm{ord}(g_0)$ denotes the order of $g_0$.
\end{fact}

\begin{example}
\label{ex:(1/2,1/2,1)}
If $(a,b,c)=(\frac{1}{2},\frac{1}{2},1)$ then we have 
$\big(\frac{1}{|1-c|},\frac{1}{|c-a-b|},\frac{1}{|a-b|}\big)=
(\infty,\infty,\infty)$, 
$$
M_0=\begin{pmatrix}
1 & -2 \\
0 & 1
\end{pmatrix},\quad 
M_1
=\begin{pmatrix}
1 & 0 \\
2 & 1
\end{pmatrix},\quad 
M_0^{-1}M_1^{-1}
=\begin{pmatrix}
-3 & 2 \\
-2 & 1
\end{pmatrix}.
$$
The group generated by these matrices is not 
the principal congruence subgroup $\Gamma(2)$ of level $2$ in $\mathrm{SL}_2(\Z)$ but
$$\Gamma(2,4)=\{g=(g_{ij})\in \mathrm{SL}_2(\Z)\mid g_{11},g_{22}\equiv 1\bmod 4, 
g_{12},g_{21}\equiv 0\bmod 2\}.$$
Note that $\Gamma(2,4)$ is a subgroup in $\Gamma(2)$ of index $2$, 
and that $\pm I_2$ are representatives of the quotient 
$\Gamma(2)/\Gamma(2,4)$. 

Since the image of the monodromy representation is in $\mathrm{PSL}_2(\R)$ 
and $f_1(z)/f_2(z)$ is in $\H$ for $z\in (0,1)$, 
the image of the analytic continuation of $f_1(z)/f_2(z)$ to $Z=\C-\{0,1\}$ 
is the whole space $\H$. 
\end{example}

\begin{example}
\label{ex:(1/4,1/4,1)}
If $(a,b,c)=(\frac{1}{4},\frac{1}{4},1)$ then we have 
$\big(\frac{1}{|1-c|},\frac{1}{|c-a-b|},\frac{1}{|a-b|}\big)=
(\infty,2,\infty)$, 
$$
M_0
=\begin{pmatrix}
1 & -1-\i \\
0 & 1
\end{pmatrix},\quad 
M_1
=\begin{pmatrix}
-1 & 0 \\
1+\i & 1
\end{pmatrix},\quad 
M_0^{-1}M_1^{-1}
=\begin{pmatrix}
-1+2\i & 1+\i \\
1+\i & 1
\end{pmatrix}.
$$
By using 
$$P=\begin{pmatrix}
-1+\i & \i\\ 
0& 1
\end{pmatrix}, 
$$
we change the basis 
$\mathbf{F}(z)$ into $P\mathbf{F}(z)$.  
Then 
$$
PM_0P^{-1}
=\begin{pmatrix}
1 & 2 \\
0 & 1
\end{pmatrix},\quad 
PM_1P^{-1}=\begin{pmatrix}
0 & \i \\
-\i & 0
\end{pmatrix}=\i J,\quad 
PM_0^{-1}M_1^{-1}P^{-1}=\begin{pmatrix}
2\i & \i \\
-\i & 0
\end{pmatrix}.
$$
The group generated by these matrices is 
$$
\Gamma(2,4)\la \i J\ra
=
\{g\in \mathrm{GL}_2(\Z[\i])\mid g\in \Gamma(2,4) \textrm{ or } 
(\i J)g\in \Gamma(2,4)\}=\Gamma(2,4)\cup (\i J)\cdot \Gamma(2,4)
$$
since 
$$\begin{pmatrix}
0 & \i \\
-\i & 0
\end{pmatrix}
\begin{pmatrix}
1 & 2 \\
0 & 1
\end{pmatrix}
\begin{pmatrix}
0 & \i \\
-\i & 0
\end{pmatrix}=\begin{pmatrix}
1 & 0 \\
-2 & 1
\end{pmatrix}. 
$$
Note that the projectivization of $\Gamma(2,4)\la \i J\ra$ is 
isomorphic to that of 
$$\Gamma_{12}=\{g=(g_{ij})\in \mathrm{SL}_2(\Z)\mid g_{11}g_{12}\equiv 
g_{21}g_{22}\equiv 0 \bmod 2\}.
$$

Since the image of the monodromy representation is in $\mathrm{PSL}_2(\R)$ 
and $\big((-1+\i)f_1(z)+\i f_2(z)\big)/f_2(z)$ is in $\H$ for $z\in (0,1)$, 
the image of the analytic continuation of $f_2(z)/f_1(z)$ to $Z=\C-\{0,1\}$ 
is a open dense subset of $\H$. 
\end{example}

It is known that $F(a,b,c;z)$ satisfies several kinds of 
transformation formulas with respect to variable changes. 
In this paper, we use the following.

\begin{fact}[{\cite[(24) in p.64, (18),(19) in p.112 of Vol.I]{Er}}]
\begin{align}
\label{eq:FE-Gauss1}
F(a,b,2b;\frac{4z}{(1+z)^2})&
=(1+z)^{2a}F(a,a-b+\frac{1}{2},b+\frac{1}{2};z^2),\\
\label{eq:FE-Gauss2}
F(a,b,\frac{a+b+1}{2};z)&=F(\frac{a}{2},\frac{b}{2},\frac{a+b+1}{2};4z(1-z))\\
\label{eq:FE-Gauss3}
&=(1-2z)F(\frac{a+1}{2},\frac{b+1}{2},\frac{a+b+1}{2};4z(1-z)),
\end{align}
where $z$ is sufficiently near to $0$ and 
$(1+z)^{2a}$ takes the value $1$ at $z=0$.
\end{fact}

\section{Fundamental properties of $\h_{pq}(\tau)$}
We next introduce theta constants and their properties by referring to 
\cite{Er} and \cite{Mu}.
\begin{definition}
The theta constant $\h_{pq}(\tau)$ with characteristics $p,q$  
is defined by 
\begin{equation}
\label{eq:theta}
\h_{pq}(\tau)=\sum_{n=-\infty}^\infty \exp\big(\pi\i(n+\frac{p}{2})^2\tau
+2\pi\i(n+\frac{p}{2})\frac{q}{2}\big),
\end{equation}
where 
$\tau$ is a variable in the upper half plane 
$\H$, and $p,q$ are parameters taking the value $0$ or $1$.
It converges absolutely and uniformly on any compact set in 
$\H$ for any fixed $p,q$.
\end{definition}

There are four functions $\h_{00}(\tau)$, $\h_{01}(\tau)$, 
$\h_{10}(\tau)$, $\h_{11}(\tau)$; 
the last one vanishes identically on $\H$, and the rests 
satisfy  Jacobi's identity
\begin{equation}
\label{eq:J-Id}
\h_{00}(\tau)^4=\h_{01}(\tau)^4+\h_{10}(\tau)^4
\end{equation}
for any $\tau\in \H$. We have twice formulas of them.
\begin{fact}[{\cite[(15) in p.373 of Vol.II]{Er}}]  
\label{fact:2tau}
The theta constants satisfy twice formulas 
\begin{equation} 
\label{eq:2tau}
\begin{array}{l}
\h_{00}(2\tau)^2=\dfrac{\h_{00}(\tau)^2+\h_{01}(\tau)^2}{2},\\[2mm]
\h_{01}(2\tau)^2=\h_{00}(\tau)\h_{01}(\tau),\\[2mm]
\h_{10}(2\tau)^2=\dfrac{\h_{00}(\tau)^2-\h_{01}(\tau)^2}{2}.
\end{array}
\end{equation}
\end{fact}

\begin{fact}[{\cite[Theorem 7.1]{Mu}}]
\label{fact:trans-theta}
Let $g=(g_{ij})$ be an element in the projectivization $\mathrm{P}\Gamma_{12}$
of $\Gamma_{12}$. By multiplying $-1$ to $g$ if necessary, 
we may assume that $g$ satisfies either 
$g_{21}>0$, or 
$g_{21}=0$ and $g_{22}>0$. 
Then we have 
$$\h_{00}(g\cdot \tau)^2
=\chi(g)\cdot (g_{21}\tau+g_{22})\h_{00}(\tau)^2
$$
for any $\tau\in \D_{12}$, where 
$$\chi(g)=\left\{
\begin{array}{lcc} \i^{g_{22}-1} &\textrm{if} & g_{21}\in 2\Z,
g_{22}\notin 2\Z,\\
\i^{-g_{21}} &\textrm{if} & g_{21}\notin 2\Z,g_{22}\in 2\Z.
\end{array}
\right.
$$
\end{fact}
\begin{fact}[{\cite[Table V in p.36]{Mu}}]
\label{fact:J-act}
We have 
$$\h_{00}(\frac{-1}{\tau})^2=(-\i\tau)\h_{00}(\tau)^2,\quad 
\h_{01}(\frac{-1}{\tau})^2=(-\i\tau)\h_{10}(\tau)^2,\quad 
\h_{10}(\frac{-1}{\tau})^2=(-\i\tau)\h_{01}(\tau)^2.
$$
\end{fact}

\begin{fact}[The inverse of Schwarz's map]
\label{fact:Inv-S-map}
\begin{enumerate}
\item \label{item:lambda}
\cite[\S 5.6,5.7,5.8 in Chap. II, Proposition 8.1 in Chap. III]{Yo}
The inverse of Schwarz's map 
\begin{equation}
\label{eq:S-Map-1/2}
z\mapsto \tau=\frac{f_1(z)}{f_2(z)}=\i\cdot 
\frac{F(\frac{1}{2},\frac{1}{2},1,1-z)}{F(\frac{1}{2},\frac{1}{2},1,z)}
\end{equation} 
for $(a,b,c)=(\frac{1}{2},\frac{1}{2},1)$ is 
$$\H\ni \tau \mapsto \lambda(\tau)=\frac{\h_{10}(\tau)^4}
{\h_{00}(\tau)^4}\in Z.$$ 
The function $\lambda(\tau)$ is invariant under the action 
$\tau\mapsto (g_{11}\tau+g_{12})/(g_{21}\tau+g_{22})$ 
of $g=(g_{ij})\in \Gamma(2)$ on $\tau \in \H$.
\item \label{item:zeta}
\cite[Theorem 8.4]{MOT}
The inverse of Schwarz's map 
\begin{align}
\nonumber
z\mapsto \tau=&\frac{(-1+\i)f_1(z)+\i f_2(z)}{f_2(z)}
\\
\label{eq:S-Map-1/4}
=&\i\cdot \frac{\pi F(\frac{1}{4},\frac{1}{4},1;z)
+B(\frac{3}{4},\frac{3}{4})
\sqrt{1-z}F(\frac{3}{4},\frac{3}{4},\frac{3}{2},1-z)}
{\pi F(\frac{1}{4},\frac{1}{4},1,z)}
\\
\nonumber
=&\i\cdot \frac{-\pi F(\frac{1}{4},\frac{1}{4},1;z)
+B(\frac{1}{4},\frac{1}{4})
F(\frac{1}{4},\frac{1}{4},\frac{1}{2},1-z)}
{\pi F(\frac{1}{4},\frac{1}{4},1,z)},
\end{align}
for $(a,b,c)=(\frac{1}{4},\frac{1}{4},1)$ is 
$$
\H\ni \tau \mapsto \zeta(\tau)
=\frac{4\h_{01}^4(\tau)\h_{10}^4(\tau)}{\h_{00}(\tau)^8}\in Z\cup\{1\}
=\C-\{0\}.
$$
The function $\zeta(\tau)$ is invariant under the action 
$\tau\mapsto (g_{11}\tau+g_{12})/(g_{21}\tau+g_{22})$ 
of $g=(g_{ij})\in \Gamma_{1,2}$ on $\tau \in \H$.
\end{enumerate}
\end{fact}

Here we give a fundamental region of $\Gamma(2)$ and that of $\Gamma_{12}$ as
\begin{align}
\label{eq:fund-reg-G(2)}
\D(2)&
=\{\tau\in \H\mid -1<\re(\tau)\le1,\ |\tau-\dfrac{1}{2}|\ge \dfrac{1}{2},
\ |\tau+\dfrac{1}{2}|> \dfrac{1}{2}
\},
\\
\label{eq:fund-reg-G01(2)}
\D_{12}&
=\{\tau\in \H\mid -1< \re(\tau)\le 1,\ |\tau|> 1\}\cup 
\{\tau \in \H \mid |\tau|=1, 0\le \re(\tau)\le 1\}.
\end{align}
\begin{figure}[htb]

\includegraphics[width=10cm]{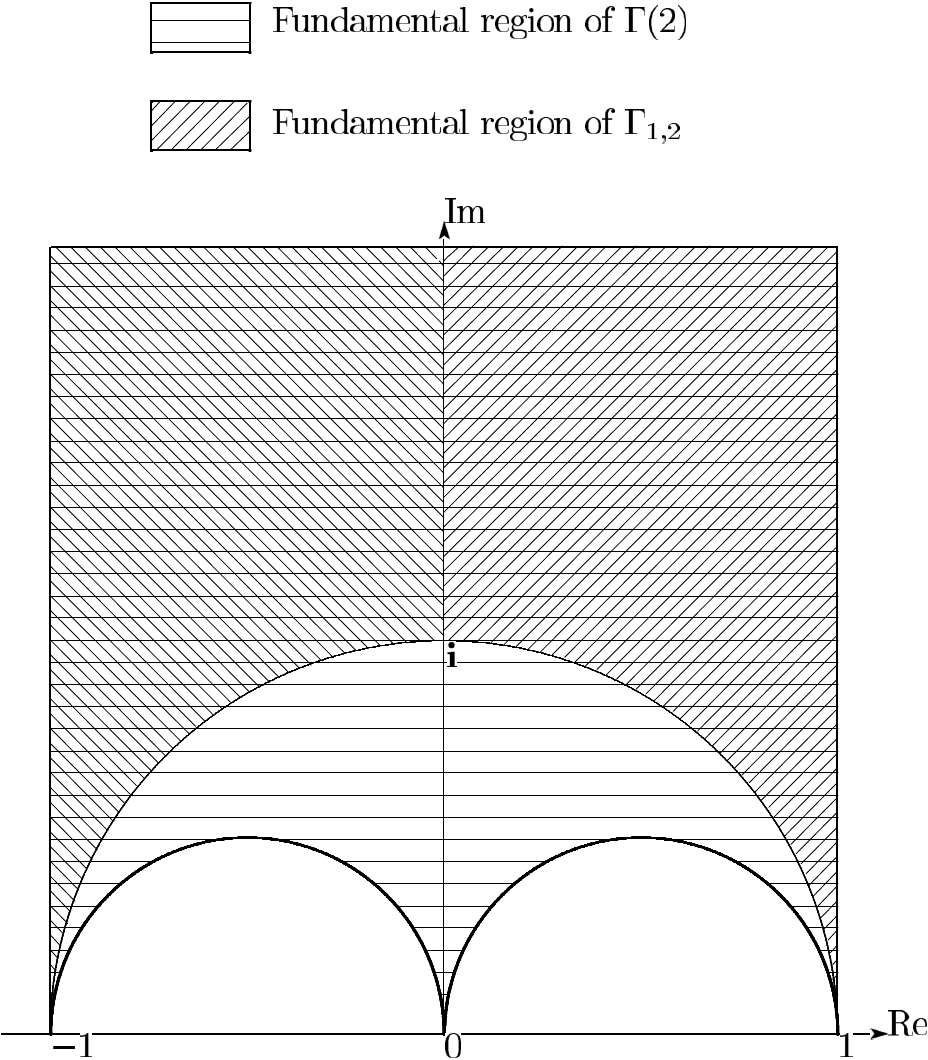}
\caption{Fundamental regions for $\Gamma(2)$ and $\Gamma_{12}$} 
\end{figure}

\section{Jacobi's formula and its analogy}
As given in \cite[Theorem 2.1]{BB}, 
the hypergeometric series $F(\frac{1}{2},\frac{1}{2},1;x)$ and 
the theta constants are related by Jacobi's formula:
\begin{fact}[Jacobi's formula]
\label{fact:Jacobi}
Set 
$$\lambda(\tau)=\frac{\h_{10}(\tau)^4}{\h_{00}(\tau)^4}$$
for any element $\tau$ in 
$\D(2)$.  
Then Jacobi's formula
\begin{equation}
\label{eq:Jacobi}
F(\frac{1}{2},\frac{1}{2},1;\lambda(\tau))=\h_{00}(\tau)^2
\end{equation}
holds, where $F(\frac{1}{2},\frac{1}{2},1;z)$ denotes  
the single-valued function on $\C-\{1\}$ given 
in Remark \ref{rem:discontinuous}. 
By the restriction of Schwarz's map 
$$
 z\mapsto \tau=\tau(z)=\i\frac{F(\frac{1}{2},\frac{1}{2},1;1-z)}
 {F(\frac{1}{2},\frac{1}{2},1;z)}
 $$
in \eqref{eq:S-Map-1/2} to the domain $\{z\in \C\mid |z|<1,\ |z-1|<1\}$, 
the equality \eqref {eq:Jacobi} is pulled back to  
\begin{equation}
\label{eq:Jacobi-z}
F(\frac{1}{2},\frac{1}{2},1;z)=
\theta_{00}(\tau(z))^2.
\end{equation}
\end{fact}

\begin{remark}
\label{rem:boundary}
Suppose that $\tau$ is in the interior $\D(2)^\circ$ of $\D(2)$. 
\begin{enumerate}
\item If $\re(\tau)\ge 0$ then $\im(\lambda(\tau))\ge 0$, 
otherwise $\im(\lambda(\tau))<0$. 

\item 
If $\tau$ approaches to a point in the boundary components 
$\{\tau\in \H\mid \re(\tau)=\pm 1\}$ of $\D(2)$, 
then $\lambda(\tau)$ converges to a negative real value, at which 
the hypergeometric series is naturally extended.

\item 
If $\tau$ 
approaches to a point in the boundary component 
$\{\tau\in \H\mid |\tau-\frac{1}{2}|=\frac{1}{2}\}$ of $\D(2)$,  
then $\lambda(\tau)$ approaches to a real value greater than $1$ 
via the upper half plane. 
Note that the addition of this component to $\D(2)$ 
is compatible with the definition of $F(a,b,c;z)$ on $\C-\{1\}$ 
in Remark \ref{rem:discontinuous}. 

\end{enumerate}

\end{remark}

To extend \eqref{eq:Jacobi}
to the formula on the whole space $\H$, 
we need to consider the continuation of 
$F(\frac{1}{2},\frac{1}{2},1;\lambda(\tau))$.  
We prepare a lemma.
\begin{lemma}
\label{lem:factorG2}
For
$g=(g_{ij})\in \Gamma(2,4)\la \i J\ra$,  
let $g'=(g'_{ij})\in \mathrm{P}\Gamma_{12}$ be its normalization of $g$ in  
the formula of Fact \ref{fact:trans-theta}. 
Then we have 
$$g_{21}\tau+g_{22}=\chi(g')\cdot (g'_{21}\tau+g'_{22}).$$
\end{lemma}

\begin{proof} 
At first, we consider the case $g=(g_{ij})\in \Gamma(2,4)$. 
In this case, we have $g_{22}\equiv 1\bmod 4$. 
If $g_{21}>0$ then $g'=g$ and 
$$\chi(g')\cdot (g'_{21}\tau+g'_{22})=\i^{g_{22}-1} \cdot (g_{21}\tau+g_{22})
=g_{21}\tau+g_{22}.$$
If $g_{21}<0$ then $g'=-g$ and 
$$\chi(g')\cdot (g'_{21}\tau+g'_{22})
=\i^{-g_{22}-1}\cdot (-g_{21}\tau-g_{22}) =g_{21}\tau+g_{22}.$$
If $g_{21}=0$ then $g'_{22}=|g_{22}|$ and 
$$\chi(g')\cdot (g'_{21}\tau+g'_{22})=\i^{|g_{22}|-1} \cdot |g_{22}|
=\left\{
\begin{array}{ccc}
\i^{g_{22}-1} \cdot g_{22}=g_{22} &\textrm{if}& g_{22}>0,\\
\i^{-g_{22}-1}\cdot(-g_{22}) =g_{22} &\textrm{if}& g_{22}<0.
\end{array}
\right.
$$

Next we consider the case $g=(g_{ij})\in (\i J)\cdot \Gamma(2,4)$.
In this case, $g_{21}$ takes the form of $(-\i)\cdot(4m+1)$ ($m\in \Z$), 
and $g'$ is given by a scalar multiplication of 
$\pm \i$ to $g$ with $g_{21}'=|4m+1|$.
Since they do not vanish, it is sufficient to show $\chi(g')\cdot g'_{21}=g_{12}$.   
In fact, we have
$$\chi(g')\cdot g'_{21}=|4m+1|\cdot \i^{-|4m+1|}=
\left\{ \begin{array}{rll}
(4m+1)\cdot \i^{-(4m+1)}=(-\i)\cdot (4m+1) &\textrm{if}& 4m+1>0, \\
-(4m+1)\cdot \i^{4m+1}=(-\i)\cdot (4m+1) &\textrm{if}& 4m+1<0.
        \end{array}
\right.
$$



Thus $g_{21}\tau+g_{22}=\chi(g')\cdot 
(g'_{21}\tau+g'_{22})$ holds in any cases. 
\end{proof}

\begin{cor}
\label{cor:Jacobi-formula}
We extend Jacobi's formula \eqref{eq:Jacobi} to 
the equality on the whole space $\H$ as 
\begin{align}
\label{eq:ex-Jacobi1}
\h_{00}(\tau)^2=&F_\rho(\frac{1}{2},\frac{1}{2},1;\lambda(\tau))\\
\label{eq:ex-Jacobi2}
=&(g_{21}\tau_0+g_{22})F(\frac{1}{2},\frac{1}{2},1;\lambda(\tau_0))
=\frac{F(\frac{1}{2},\frac{1}{2},1;\lambda(\tau))}
{-g_{21}\tau+g_{11}},
\end{align}
where $\tau$ is any element in $\H$, $\tau_0\in \D(2)$ and 
$g=(g_{ij})\in \Gamma(2,4)$ satisfy
$$\tau=g\cdot \tau_0=\frac{g_{11}\tau_0+g_{12}}{g_{21}\tau_0+g_{22}},$$
$\rho$ is the image of a path from 
$\tau_0$ to $\tau$ in $\H$ under the map 
$\H\ni \tau \mapsto \lambda(\tau)\in \C-\{0,1\}$, 
$F(\frac{1}{2},\frac{1}{2},1;z)$ denotes 
the single-valued function on $\C-\{1\}$ given 
in Remark \ref{rem:discontinuous}, 
and $F_\rho(\frac{1}{2},\frac{1}{2},1;w)$ is its analytic continuation 
along $\rho$. 

\end{cor}
\begin{proof}
We obtain the identity \eqref{eq:ex-Jacobi1}  
by the analytic continuation of \eqref{eq:Jacobi} to $\H$. 
We also have 
\begin{align*}
\h_{00}(\tau)^2=&\h_{00}(g\cdot \tau_0)^2
=(g_{21}\tau_0+g_{22})\cdot \h_{00}(\tau_0)^2
=(g_{21}\tau_0+g_{22})\cdot F(\frac{1}{2},\frac{1}{2},1;\lambda(\tau_0))\\
=&(g_{21}\tau_0+g_{22})\cdot F(\frac{1}{2},\frac{1}{2},1;\lambda(\tau))
=\frac{F(\frac{1}{2},\frac{1}{2},1;\lambda(\tau))}{-g_{21}\tau+g_{11}} 
\end{align*}
by \eqref{eq:Jacobi} and Lemma \ref {lem:factorG2}.
Here note that 
$$
\tau_0=\frac{g_{22}\tau-g_{12}}{-g_{21}\tau+g_{11}},\quad 
g_{21}\tau_0+g_{22}=g_{21}\frac{g_{22}\tau-g_{12}}{-g_{21}\tau+g_{11}}
+g_{22}=\frac{1}{-g_{21}\tau+g_{11}},
$$
and that $\lambda(\tau)=\lambda(\tau_0)$. 
\end{proof}

We give an analogy of Jacobi's formula \eqref{eq:Jacobi}. 
\begin{theorem}
\label{th:theta=HGS}
We set 
$$\zeta(\tau)=\frac{4\h_{01}(\tau)^4\h_{10}(\tau)^4}{\h_{00}(\tau)^8}$$
for any $\tau$ in 
$\D_{12}$.
Then we have 
\begin{align}
\label{eq:1/4}
F\Big(\frac{1}{4},\frac{1}{4},1;\zeta(\tau)\Big)
=&\h_{00}(\tau)^2,\\
\label{eq:3/4}
F\Big(\frac{3}{4},\frac{3}{4},1;\zeta(\tau)\Big)
=&\frac{\h_{00}(\tau)^6}{\h_{01}(\tau)^4-\h_{10}(\tau)^4},
\end{align}
where the hypergeometric series is extended to 
the single-valued function on $\C-\{1\}$ given 
in Remark \ref{rem:discontinuous}. 
\end{theorem}
\begin{proof}
By substituting $a=b=1/2$ into \eqref{eq:FE-Gauss2},
we have 
$$F(\frac{1}{2},\frac{1}{2},1;\lambda)=F(\frac{1}{4},\frac{1}{4},1;
4\lambda(1-\lambda)),$$
where $\lambda$ is sufficiently near to $0$. 
Jacobi's formula \eqref{eq:Jacobi} together with 
Jacobi's identity \eqref{eq:J-Id}
yields that 
\begin{align*}
\h_{00}(\tau)^2&=
F\Big(\frac{1}{2},\frac{1}{2},1;\frac{\h_{10}(\tau)^4}{\h_{00}(\tau)^4}\Big)
=F\Big(\frac{1}{4},\frac{1}{4},1;4\frac{\h_{10}(\tau)^4}{\h_{00}(\tau)^4}
\big(1-\frac{\h_{10}(\tau)^4}{\h_{00}(\tau)^4})\Big)\\
&=F\Big(\frac{1}{4},\frac{1}{4},1;\frac{4\h_{01}(\tau)^4\h_{10}(\tau)^4}
{\h_{00}(\tau)^8}\Big)
\end{align*}
for $\tau$ with sufficiently large imaginary part. 
We can extend this identity to any $\tau\in \D_{12}$  since 
if $\tau$ belongs to 
the interior of 
$\D_{12}$ then $\zeta(\tau)$ is in 
$\C-[1,\infty)$ 
on which $F(1/4,1/4,1;z)$ is single valued. 
 
To show the second formula, use \eqref{eq:FE-Gauss3} with 
the substitution $a=b=1/2$. Then we have 
$$\h_{00}(\tau)^2=\big(1-2\frac{\h_{10}(\tau)^4}{\h_{00}(\tau)^4}\big)
F\Big(\frac{3}{4},\frac{3}{4},1;\zeta(\tau)\Big),
$$
which is equivalent to \eqref{eq:3/4}.
\end{proof}

\begin{remark}
\begin{enumerate}
\item Substitute $\tau=\i$ into \eqref{eq:1/4}.
Fact \ref{fact:J-act} implies $\h_{01}(\i)^2=\h_{10}(\i)^2$,  
and Jacobi's identity \eqref{eq:J-Id} becomes  
$\h_{00}(\i)^4=\h_{01}(\i)^4+\h_{10}(\i)^4=2\h_{10}(\i)^4$.
By Gauss-Kummer's identity \eqref{eq:G-K}, we have
$$
\h_{00}(\i)^2=F(\frac{1}{4},\frac{1}{4},1;\frac{4\h_{01}(\i)^4\h_{10}(\i)^4}
{\h_{00}(\i)^8})=F(\frac{1}{4},\frac{1}{4},1;1)
=\frac{\G(1)\G(\frac{1}{2})}{\G(\frac{3}{4})
\G(\frac{3}{4})}=\Big(\frac{\pi^{1/4}}{\G(\frac{3}{4})}\Big)^2,\\
$$
which yields 
$$\h_{00}(\i)=\frac{\pi^{1/4}}{\G(\frac{3}{4})}$$
since $\h_{00}(\i)$ is positive real.

\item 
The right hand side of \eqref{eq:3/4} has a simple pole at $\tau=\i$.
Here note that we cannot apply 
Gauss-Kummer's identity \eqref{eq:G-K} to 
the left hand side of \eqref{eq:3/4} for $\tau=\i$
since $c-a-b$ for $(a,b,c)=(\frac{3}{4},\frac{3}{4},1)$ is 
$1-\frac{3}{4}-\frac{3}{4}=-\frac{1}{2}<0$.
\end{enumerate}
\end{remark}

We extend Theorem \ref{th:theta=HGS} to formulas for 
the whole space $\H$.

\begin{cor}
\label{cor:extend-J-1/4}
The equalities \eqref{eq:1/4} and \eqref{eq:3/4} are extended to 
\begin{align}
\label{eq:ex-Jacobi1/4}
&\hspace{1.8cm}\begin{array}{ll}
\h_{00}(\tau)^2
&=F_\rho(\frac{1}{4},\frac{1}{4},1;\zeta(\tau))\\
&=(g_{21}\tau_0+g_{22})F(\frac{1}{4},\frac{1}{4},1;\zeta(\tau_0))
=\dfrac{\det(g) F(\frac{1}{4},\frac{1}{4},1;\zeta(\tau))}{-g_{21}\tau+g_{11}},\\
\end{array}
\\
\label{eq:ex-Jacobi3/4}
&\begin{array}{ll}
\dfrac{\h_{00}(\tau)^6}{\h_{01}(\tau)^4-\h_{10}(\tau)^4}
&=F_\rho(\frac{3}{4},\frac{3}{4},1;\zeta(\tau))\\
&=\dfrac{g_{21}\tau_0+g_{22}}{\det(g)}
F(\frac{3}{4},\frac{3}{4},1;\zeta(\tau_0))
=\dfrac{F(\frac{3}{4},\frac{3}{4},1;\zeta(\tau))}
{-g_{21}\tau+g_{11}},
\end{array}
\end{align}
where $\tau$ is any element in $\H$, $\tau_0\in \D_{12}$ and 
$g=(g_{ij})\in \Gamma(2,4)\la \i J\ra$ 
satisfy 
$$\tau=g\cdot \tau_0=\frac{g_{11}\tau_0+g_{12}}{g_{21}\tau_0+g_{22}},
$$
$\rho$ is the image of a path $\tau_0$ to $\tau$ in $\H$ under the 
map $\H\ni\tau \mapsto \zeta(\tau)\in \C-\{1\}$, 
$F(a,b,c;z)$ is extended to 
the single-valued function on $\C-\{1\}$ given in Remark 
\ref{rem:discontinuous}, and 
$F_\rho(a,b,c;w)$ is its analytic continuation along $\rho$.
\end{cor}
\begin{proof} 
We can show 
$$\h_{00}(\tau)^2
=F_\rho(\frac{1}{4},\frac{1}{4},1;\zeta(\tau))
=(g_{21}\tau_0+g_{22})F(\frac{1}{4},\frac{1}{4},1;\zeta(\tau_0))
$$
quite similarly to Proof of Corollary \ref{cor:Jacobi-formula}.  
Since 
$$\det(g)=\left\{\begin{array}{rcl} 
 1 &\textrm{if} & g\in \Gamma(2,4),\\
-1 &\textrm{if} & g\in (\i J)\cdot \Gamma(2,4),
 \end{array}
\right.
$$
we have 
$$g_{21}\tau_0+g_{22}=\frac{1}{\det(g)(-g_{21}\tau+g_{11})}
$$
for any $g\in \Gamma(2,4)\la \i J\ra$, 
which leads to the last equality in \eqref{eq:ex-Jacobi1/4}.

To show \eqref{eq:ex-Jacobi3/4}, we have only to note that 
\begin{align*}
\frac{\h_{00}(\tau)^6}{\h_{01}(\tau)^4-\h_{10}(\tau)^4}
&=\h_{00}(g\cdot \tau_0)^2
\frac{\h_{00}(g\cdot \tau_0)^4}
{\h_{01}(g\cdot \tau_0)^4-\h_{10}(g\cdot \tau_0)^4}, \\
\h_{00}(g\cdot \tau_0)^2&=(g_{21}\tau_0+g_{22})\h_{00}(\tau_0)^2, \\
\frac{\h_{00}(g\cdot \tau_0)^4}
{\h_{01}(g\cdot \tau_0)^4-\h_{10}(g\cdot \tau_0)^4}
&=
\left\{\begin{array}{ccl}
\dfrac{\h_{00}(\tau_0)^4}{\h_{01}(\tau_0)^4-\h_{10}(\tau_0)^4}
&\textrm{if} & g\in \Gamma(2,4),\\[4mm]
\dfrac{\h_{00}(\tau_0)^4}{\h_{10}(\tau_0)^4-\h_{01}(\tau_0)^4}
&\textrm{if} & g\in (\i J)\cdot \Gamma(2,4),\\
         \end{array}
 \right.
\end{align*}
by \eqref{eq:J-Id} and  Facts \ref{fact:trans-theta}, \ref{fact:J-act},  
 \ref{fact:Inv-S-map} (\ref{item:lambda}).
\end{proof}


By regarding $\tau$ as a dependent variable of $z$ in 
the domain $\{z\in \C\mid |z|<1,\ |z-1|<1\}$
by  Schwarz's map \eqref{eq:S-Map-1/4}, 
we rewrite Theorem \ref{th:theta=HGS} into the following. 
\begin{cor}
By the restriction of Schwarz's map 
$z\mapsto \tau(z)$ in \eqref{eq:S-Map-1/4} 
to the domain $\{z\in \C\mid |z|<1,\ |z-1|<1\}$,   
the equalities \eqref{eq:1/4}, \eqref{eq:3/4} are pulled back to 
$$F(\frac{1}{4},\frac{1}{4},1;z)=
\h_{00}(\tau(z))^2,\quad 
F(\frac{3}{4},\frac{3}{4},1;z)=
\frac{\h_{00}(\tau(z))^6}{\h_{01}(\tau(z))^4-\h_{10}(\tau(z))^4}.
$$
\end{cor}

\section{A transformation formula for $F(\frac{1}{4},\frac{1}{4},1;z)$.}
\begin{theorem}
\label{th:trans}
We have a transformation formula 
\begin{equation}
\label{eq:FE:HG}
\frac{2+\sqrt{2+2w}}{4}
F(\frac{1}{4},\frac{1}{4},1;1-w^2)
=F\Big(\frac{1}{4},\frac{1}{4},1;1-
\frac{(6\sqrt{2w+2}-w-3)^2}{(2\sqrt{2w+2}+w+3)^2}\Big),
\end{equation}
where $w$ is in a neighborhood of $1$, and $\sqrt{2+2w}$ takes $2$ at $w=1$.

\end{theorem}
\begin{proof} 
By substituting $2\tau$ into $\tau$ for \eqref{eq:1/4}, 
and using \eqref{eq:2tau}, we have
\begin{equation}
\label{eq:key-id}
F(\frac{1}{4},\frac{1}{4},1;\zeta(2\tau))=\h_{00}(2\tau)^2=
\frac{\h_{00}(\tau)^2+\h_{01}(\tau)^2}{2}
=\frac{\h_{00}(\tau)^2+\h_{01}(\tau)^2}{2\h_{00}(\tau)^2}\cdot \h_{00}(\tau)^2,
\end{equation} 
where we restrict $\tau\in \H$ to pure imaginary numbers with 
sufficiently large imaginary part. 
Note that 
\begin{align*}
\zeta(\tau)&=\frac{4\h_{01}(\tau)^4\h_{10}(\tau)^4}{\h_{00}(\tau)^8}
=1-\frac{(\h_{01}(\tau)^4+\h_{10}(\tau)^4)^2-4\h_{01}(\tau)^4\h_{10}(\tau)^4}
{\h_{00}(\tau)^8}\\
&=1-\frac{(\h_{01}(\tau)^4-\h_{10}(\tau)^4)^2}{\h_{00}(\tau)^8}
=1-\Big(\frac{2\h_{01}(\tau)^4-\h_{00}(\tau)^4}{\h_{00}(\tau)^4}\Big)^2.
\end{align*}
Put 
$$w=\frac{2\h_{01}(\tau)^4-\h_{00}(\tau)^4}{\h_{00}(\tau)^4},
$$
then we have
$$\h_{01}(\tau)^4=\frac{w+1}{2}\cdot \h_{00}(\tau)^4,
\quad \h_{01}(\tau)^2=\frac{\sqrt{2w+2}}{2}\cdot \h_{00}(\tau)^2,
$$
the last term of \eqref{eq:key-id} is 
$$
\frac{\h_{00}(\tau)^2+\h_{01}(\tau)^2}{2\h_{00}(\tau)^2}\cdot \h_{00}(\tau)^2
=
\frac{2+\sqrt{2w+2}}{4}\cdot F(\frac{1}{4},\frac{1}{4},1;1-w^2),
$$
which is the left hand side of \eqref{eq:FE:HG}. 
Here note that $\h_{01}(\tau)>0$ for any pure imaginary number $\tau$.  

Let us express $\zeta(2\tau)$ in terms of $w$. We have 
\begin{align*}
\zeta(2\tau)&
=1-\Big(\frac{2\h_{01}(2\tau)^4-\h_{00}(2\tau)^4}{\h_{00}(2\tau)^4}\Big)^2
=1-\Big(\frac{8\h_{00}(\tau)^2\h_{01}(\tau)^2-(\h_{00}(\tau)^2+\h_{01}(\tau)^2)^2}
{\h_{00}(\tau)^2+\h_{01}(\tau)^2}\Big)^2\\
&=1-\Big(\frac{16\sqrt{2w+2}-(2+\sqrt{2w+2})^2}{(2+\sqrt{2w+2})^2}\Big)^2
=1-\Big(\frac{6\sqrt{2w+2}-w-3}{2\sqrt{2w+2}+w+3}\Big)^2.
\end{align*}
Hence $F(\frac{1}{4},\frac{1}{4},1;\zeta(2\tau))$ 
in \eqref{eq:key-id} is equal to 
the right hand side of \eqref{eq:FE:HG}.
\end{proof}


\section{Mean iterations}
By referring to \cite[Chapter 8]{BB}, 
we give fundamental properties of mean iterations. 
\begin{definition}[Means]
\label{def:mean}
A mean on a subset $\sS$ in $\R^2$ is a continuous function $m$ defined  on
$\sS$  
satisfying 
\begin{equation}
\label{eq:def-mean}
\min(x,y)\le m(x,y)\le \max(x,y)
\end{equation}
for any $(x,y)\in \sS$. \\
A mean $m(x,y)$ is strict if it satisfies 
\begin{equation}
\label{eq:def-strict}
m(x,y)=x\textrm{ or }m(x,y)=y \quad \Leftrightarrow \quad x=y.
\end{equation}
A mean $m(x,y)$ is homogeneous if $\sS$ and $m$ satisfy
\begin{equation}
\label{eq:homog}
(r\cdot x,r\cdot y)\in \sS,
\quad m(r\cdot x,r\cdot y)=r\cdot m(x,y)
\end{equation}
for any $(x,y)\in \sS$ and any $r\in \R_+=\{x\in \R\mid x>0\}$.\\
A mean $m(x,y)$ is symmetric if 
$\sS$ and $m$ satisfy 
\begin{equation}
\label{eq:symm}
(y,x)\in \sS,
\quad m(y,x)=m(x,y)
\end{equation}
for any $(x,y)\in \sS$. \\
\end{definition}
We set the arithmetic mean on $\R^2$, 
the geometric mean  on $\R_+^2$,  
and the harmonic mean on $\R_+^2$ by 
$$m_A(x,y)=\frac{x+y}{2},\quad m_G(x,y)=\sqrt{xy},\quad 
m_H(x,y)=\frac{2xy}{x+y},$$
respectively. They are strict, homogeneous and symmetric means, 
and they satisfy 
$$m_A(x,y)\ge  m_G(x,y)\ge m_H(x,y)$$
on $\R_+^2$.

We modify \cite[Proposition 8.4]{BB} as follows.
\begin{lemma}[]
\label{lem:comp}
Let $m_0$, $m_1$ and $m_2$ be means on $\sS_0$, $\sS_1$ and $\sS_2$ 
satisfying 
$$(m_1(x,y),m_2(x,y))\in \sS_0$$ for any $(x,y)\in \sS_1\cap \sS_2.$ 
Then the function 
$$m(x,y)=m_0(m_1(x,y),m_2(x,y))$$
is a mean on $\sS_1\cap \sS_2$. 
If two of $m_0,m_1,m_2$ are strict, then so is $m$. 
If $m_1$ and $m_2$ are symmetric, then so is $m$. 
If all three are homogeneous, then so is $m$. 
\end{lemma}

\begin{proof}
Note that the function $m$ is defined on $\sS_1\cap \sS_2$, and continuous 
on it. 
Since $m_0,m_1,m_2$ are means, they satisfy 
\begin{equation}
\label{eq:m0-ineq}
\min(m_1(x,y),m_2(x,y))\le m_0(m_1(x,y),m_2(x,y))\le \max(m_1(x,y),m_2(x,y)),
\end{equation}
\begin{equation}
\label{eq:m1m2-ineq}
\min(x,y)\le m_1(x,y)\le \max(x,y),\quad \min(x,y)\le m_2(x,y)\le \max(x,y).
\end{equation}
Hence we have 
\begin{equation}
\label{eq:m0m1m2-mean}
\min(x,y)\le m(x,y)\le \max(x,y),
\end{equation}
which show that $m$ is a mean on $\sS_1\cap \sS_2$. 

Suppose that $m_1$ and $m_2$ are strict. If $x\ne y$ then 
the inequalities in \eqref{eq:m1m2-ineq} are strict,  
which imply that the inequalities in \eqref{eq:m0m1m2-mean} become strict. 
Suppose that $m_0$ and $m_i$ ($i=1$ or $i=2$) are strict. 
If $(x,y)\in \sS_1\cap \sS_2$ satisfies $x\ne y$ and 
$m_1(x,y)\ne m_2(x,y)$, then
the inequalities in \eqref{eq:m0-ineq} become strict, since $m_0$ is strict.   
Thus in this case, the inequalities in \eqref{eq:m0m1m2-mean} become strict. 
If $(x,y)\in \sS_1\cap \sS_2$ satisfies $x\ne y$ and $m_1(x,y)=m_2(x,y)$ 
then 
$$m(x,y)=m_0(m_1(x,y),m_2(x,y))=m_0(m_i(x,y),m_i(x,y))=m_i(x,y).$$ 
In this case, the inequalities in \eqref{eq:m0m1m2-mean} become strict,  
since $m_i$ is strict.  
Hence it turns out that $m$ is strict if two of $m_0,m_1,m_2$ are strict. 

If each of $m_i$ ($i=1,2$) is a symmetric mean on $\sS_i$
then we have 
$$(x,y)\in \sS_1\cap\sS_2 \Rightarrow (y,x)\in \sS_1\cap\sS_2,$$
$$m(y,x)=m_0(m_1(y,x),m_2(y,x))=m_0(m_1( x,y),m_2(x,y))=m(x,y),$$
for any $(x,y)\in \sS_1\cap\sS_2$, 
which show that $m$ is symmetric.

If each of $m_i$ ($i=0,1,2$) is a homogeneous mean on $\sS_i$ then 
we have 
\begin{align*}
(x,y)&\in \sS_1\cap \sS_2\Rightarrow (r\cdot x,r\cdot y)\in \sS_1\cap \sS_2,
\\
m(r\cdot x,r\cdot y)&=m_0(m_1(r\cdot x,r\cdot y),m_2(r\cdot x,r\cdot y))
=m_0(r\cdot m_1( x,y),r\cdot m_2( x,y))\\
&=r\cdot m_0( m_1( x,y),m_2( x,y))=r\cdot m( x,y),
\end{align*}
for any $(x,y)\in \sS_1\cap \sS_2$ and any $r \in \R_+$, 
which show that $m$ is homogeneous. 
\end{proof}

\begin{definition}[A mean iteration]
Let $m_1$ and $m_2$ be means  
on $\sS_1$ and $\sS_2$ satisfying $(m_1(x,y),m_2(x,y))\in \sS_1\cap \sS_2$
for any $(x,y)\in \sS_1\cap \sS_2$. 
For any fixed $(x,y)\in \sS_1\cap \sS_2$,  
we have a pair of sequences $\{x_n\}$ and $\{y_n\}$ 
by setting $(x_0,y_0)=(x,y)$ and iterating the application of two means
\begin{equation}
\label{eq:M-iteration}
(x_{n+1},y_{n+1})=(m_1(x_n,y_n),m_2(x_n,y_n))\quad (n\in \N_0=\N\cup\{0\})
\end{equation}
to the previous terms. This construction of a pair of sequences is 
called a mean iteration.
\end{definition}

\begin{lemma}[{\cite[Theorem 8.2]{BB}}]
\label{lem:converge}
Let $m_1$ and $m_2$ be means  on $\sS_1$ and $\sS_2$
satisfying $(m_1(x,y),m_2(x,y))\in \sS_1\cap \sS_2$
for any $(x,y)\in \sS_1\cap \sS_2$. 
Suppose that the restrictions of two means $m_1$ and $m_2$ to 
$\sS_1\cap \sS_2$ are strict, and satisfy one of the following
\begin{enumerate}
\item[$(1)$] \label{item:comparable1} $m_1(x,y)\ge m_2(x,y)$ for 
any $(x,y)\in \sS_1\cap \sS_2$; 
\item[$(2)$] \label{item:comparable2} $m_1(x,y)\le m_2(x,y)$ for any 
$(x,y)\in \sS_1\cap \sS_2$; 
\item[$(3)$] \label{item:comparable3} $m_1(x,y)\le m_2(x,y)$ for $x\le y$,
 $(x,y)\in \sS_1\cap \sS_2$,  and $m_2(x,y)\le m_1(x,y)$ for $y\le x$, 
$(x,y)\in \sS_1\cap \sS_2$.  
\end{enumerate}
Then each of sequences $\{x_n\}$ and $\{y_n\}$ defined by \eqref{eq:M-iteration}
converges monotonously and uniformly on any compact set in $\sS_1\cap \sS_2$, 
and satisfies
$$\lim_{n\to \infty}x_n=\lim_{n\to \infty}y_n.$$
The function $\ds{(x,y)\mapsto \lim_{n\to \infty}x_n(=\lim_{n\to \infty}y_n)}$ 
is a strict mean on $\sS_1\cap \sS_2$, 
which is called the compound mean of $m_1$ and $m_2$ 
and denoted by $m_1\diamond m_2$.
Moreover, $m_1\diamond m_2$ is homogeneous or symmetric 
if each of $m_1$ and $m_2$ is so.
\end{lemma}

\begin{lemma}[{Invariant Principle \cite[Theorem 8.3]{BB}}]
\label{lem:IP}
Suppose that the compound mean $m_1\diamond m_2$ on $\sS_1\cap \sS_2$ 
exists for given two means $m_1$ on $\sS_1$ and $m_2$ on $\sS_2$,  
which satisfy 
$(m_1(x,y),m_2(x,y))\in \sS_1\cap \sS_2$
for any $(x,y)\in \sS_1\cap \sS_2$. 
Then it is uniquely characterized as 
a continuous function 
$\psi:\sS_1\cap \sS_2\to \R$ 
such that 
$$\psi(m_1(x,y),m_2(x,y))=\psi(x,y)$$ 
for any $(x,y)\in \sS_1\cap \sS_2$.
\end{lemma}

\begin{example}[The arithmetic-geometric mean]
The arithmetic-geometric mean $m_{AG}$ is defined by the compound mean 
$m_A\diamond m_G$ of the arithmetic mean $m_1=m_A$ on $\R^2$ and 
the geometric mean $m_2=m_G$ on $\R_+^2$; 
that is, set $(x_0,y_0)=(x,y)\in \R_+^2=\R^2\cap \R_+^2$, 
define a pair of sequences by 
$$x_{n+1}=m_A(x_n,y_n)=\frac{x_n+y_n}2,\quad 
y_{n+1}=m_G(x_n,y_n)=\sqrt{x_ny_n},$$
and take the limit 
$\ds{m_{AG}(x,y)=\lim_{n\to \infty} x_n(=\lim_{n\to \infty} y_n)}$.  

By putting $a=b=\frac{1}{2}$, $z=\frac{w-1}{w+1}$ for \eqref{eq:FE-Gauss1}, 
we have 
\begin{equation}
\label{eq:FE-Gauss}
F(\frac{1}{2},\frac{1}{2},1;1-\frac{4w}{(1+w)^2})
=\frac{1+w}{2} F(\frac{1}{2},\frac{1}{2},1;1-w^2),
\end{equation}
where $w$ is in a small neighborhood $U$ of $1\in \C$.
Substitute $w=\frac{y}{x}$ into this formula, then 
we have 
$$\frac{x} {F(\frac{1}{2},\frac{1}{2},1;1-\frac{y^2}{x^2})}
=\frac{m_A(x,y)}{F(\frac{1}{2},\frac{1}{2},1;1-\frac{m_G(x,y)^2}{m_A(x,y)^2})},
$$
which implies
\begin{equation} 
\label{eq:AGM-HGF}
m_{AG}(x,y)=\frac{x}{F(\frac{1}{2},\frac{1}{2},1;1-\frac{y^2}{x^2})}
\end{equation}
by Lemma \ref{lem:IP}. 
Since \eqref{eq:FE-Gauss} is valid on $U$, 
\eqref{eq:AGM-HGF} holds initially for $x,y$ close to each other. 
Note that the set $\{1-y^2/x^2\mid (x,y)\in \R_+^2\}$ coincides with 
the open interval $(-\infty, 1)$, on which 
$F(\frac{1}{2},\frac{1}{2},1;z)$ admits the unique analytic continuation 
satisfying \eqref{eq:FE-Gauss} as in Remark \ref{rem:discontinuous}. 
By regarding the right hand side of
\eqref{eq:AGM-HGF} as its analytic continuation, we can extend 
\eqref{eq:AGM-HGF} to a formula on $\R_+^2$. 
 \end{example}

Recall the formula \eqref{eq:FE:HG}:
$$
 \frac{2+\sqrt{2+2w}}{4}
 F(\frac{1}{4},\frac{1}{4},1;1-w^2)
 =F\Big(\frac{1}{4},\frac{1}{4},1;1-
 \frac{(6\sqrt{2w+2}-w-3)^2}{(2\sqrt{2w+2}+w+3)^2}\Big).
$$

By taking the factor in the left hand side of \eqref{eq:FE:HG}, 
we set
$$\mu_1(w)=\frac{2+\sqrt{2+2w}}{4}.$$
We also set 
$$\mu_2(w)=\frac{6\sqrt{2+2w}-w-3}{2(2+\sqrt{2+2w})}$$
so that 
$$\frac{\mu_2(w)^2}{\mu_1(w)^2}=\frac{(6\sqrt{2w+2}-w-3)^2}
{(2\sqrt{2w+2}+w+3)^2},$$ 
which is in the right hand side of \eqref{eq:FE:HG}. 
We extend them to homogeneous functions $\mu_i(x,y)$ $(i=1,2)$ 
of degree $1$ defined on $\R_+^2$ by setting $x\cdot \mu_i(y/x)$ $(i=1,2)$. 
They are 
\begin{equation}
\label{eq:mu1,2}
\mu_1(x,y)=\frac{2x+\sqrt{2x(x+y)}}{4},\quad 
\mu_2(x,y)=\frac{6\sqrt{2x(x+y)}-y-3x}{2(2x+\sqrt{2x(x+y)})}\cdot x.
\end{equation}

To understand $\mu_1$ and $\mu_2$ well, we prepare functions 
\begin{align*}
\mu_0(x,y)&=m_G(x,m_A(x,y))=\sqrt{x\cdot \frac{x+y}{2}},\\
\nu(x,y)&=2m_H(x,y)-m_A(x,y)=\frac{4xy}{x+y}-\frac{x+y}{2}
=\frac{6xy-x^2-y^2}{2(x+y)},
\end{align*}
defined on $\sS_{(-1,\infty)}=\{(x,y)\in \R^2\mid x>0,\ x+y>0\}$
and on $\R_+^2$. 
Here we introduce a notation 
$$\sS_{I}=\{(x,rx)\in \R^2\mid x\in \R_+,\ r\in I\}$$
for an interval $I$ in $\R$. For examples,  $\sS_{(0,\infty)}=\R_+^2$ and  
$$\sS_{(-1,1]}=\{(x,rx)\in \R^2\mid x\in \R_+,\ r\in (-1,1]\}
=\{(x,y)\in \R^2\mid x\in \R_+,\ -x< y\le x\}.$$ 

\begin{lemma}
\label{lem:mu0}
The function $\mu_0(x,y)$ is a positive-valued strict homogeneous mean on 
$\sS_{(-1,\infty)}$. 
It satisfies 
\begin{align}
\label{eq:ineq1}
0<y<x  
\quad  &\Rightarrow \quad 
\frac{x}{\sqrt{2}}<\mu_0(x,y)<x,\\
\label{eq:ineq2}
y<x\quad    &\Rightarrow \quad 
0<x-\mu_0(x,y)<\mu_0(x,y)-y.
\end{align}
\end{lemma}
\begin{proof}
Regard the projection $pr_x:(x,y)\mapsto x$ as a homogeneous mean on 
$\sS_{(-1,\infty)}$ and restrict the arithmetic mean $m_A$ to $\sS_{(-1,\infty)}$. 
Then they satisfy $(pr_x(x,y),m_A(x,y))\in \R_+^2$ 
for any $(x,y)\in \sS_{(-1,\infty)}$.
Since $m_A$ and $m_G$ are strict and homogeneous on $\R_+^2$,   
$\mu_0(x,y)=m_G(pr_x(x,y),m_A(x,y))$
is a positive-valued strict homogeneous mean on 
$\sS_{(-1,\infty)}$ by Lemma \ref{lem:comp}.

If $0<y<x$ then 
$$x>\mu_0(x,y)=m_G(x,m_A(x,y))>m_G(x,m_A(x,0))=m_G(x,\frac{x}{2})
=\frac{x}{\sqrt{2}}.$$
If $y<x$ then 
$$y<m_A(x,y)=\frac{x+y}{2}<\mu_0(x,y)< x,$$
which yield that 
$$0<x-\mu_0(x,y)<x-m_A(x,y)=m_A(x,y)-y<\mu_0(x,y)-y$$
even in the case $y<0$.
\end{proof}

\begin{lemma}
\label{lem:nu}
The restriction of $\nu$ to 
$\sS_{[1/3,3]}=\{(x,y)\in \R_+^2\mid \dfrac{x}{3}\le y\le 3x\}$
is a homogeneous symmetric mean, and that to 
$\sS_{(1/3,3)}=\{(x,y)\in \R_+^2\mid \dfrac{x}{3}< y< 3x\}$ 
becomes strict.
\end{lemma}
\begin{proof}
Since $m_A$ and $m_H$ are homogeneous and symmetric, 
$\nu$ satisfies 
$$\nu(r\cdot x,r\cdot y)=r\cdot \nu( x,y),\quad \nu(y,x)=\nu( x,y)$$
for any $(x,y)\in \R_+^2$ and any $r\in  \R_+$. 
Suppose that $(x,y)\in \sS_{[1/3,3]}$. 
If $x<y$ then 
\begin{align*}
\nu (x,y)-x&=\frac{4xy}{x+y}-\frac{x+y}{2}-x
=\frac{(3x-y)(y-x)}{2(x+y)}\ge 0,\\
y-\nu (x,y)&=y-\frac{4xy}{x+y}+\frac{x+y}{2}
=\frac{(3y-x)(y-x)}{2(x+y)}\ge 0,
\end{align*}
and if $x>y$ then $x-\nu (x,y)\ge 0$ and $\nu (x,y)-y\ge 0$, 
which show that the restriction of $\nu (x,y)$ to $\sS_{[1/3,3]}$ 
is a homogeneous symmetric mean. 
By these inequalities, we can also see that 
the restriction of $\nu$ to $\sS_{(1/3,3)}$ becomes strict.
\end{proof}

\begin{remark}
The function $\nu(x,y)=2m_H(x,y)-m_A(x,y)$ on $\R_+^2$ takes negative values. 
In fact, we have
$$y>(3+2\sqrt{2})x \textrm{ or }y<(3-2\sqrt{2})x \quad \Rightarrow \quad 
 \nu(x,y)<0.$$
\end{remark}

\begin{proposition}
\label{prop:mu1mu2}
\begin{enumerate}[$($i$\,)$]
\item 
\label{item:expr}
We can express the functions $\mu_1$ and $\mu_2$ by $\mu_0$ and $\nu$ as 
\begin{equation}
\label{eq:compound-ex-mu1,2}
\mu_1(x,y)=m_A(x,\mu_0(x,y)),\quad 
\mu_2(x,y)=\nu(x,\mu_0(x,y));
\end{equation}
we can regard the functions $\mu_1(x,y)$ and $\mu_2(x,y)$ 
as defined on 
$$\sS_{(-1,\infty)}=\{(x,y)\in \R^2\mid x>0,\ x+y>0\}.$$ 
\item 
\label{item:means}
The function $\mu_1$  is a strict homogeneous mean on $\sS_{(-1,\infty)}$. 
The function $\mu_2$ satisfies an inequality
\begin{equation}
\label{eq:increase-mu2}
\mu_2(x,y)\ge  y
\end{equation}
for any $(x,y)\in \sS_{(-1,1]}=\{(x,y)\in \R^2\mid x>0,\ -x<y\le x\}$.   
The restriction of $\mu_2$ to the domain 
$$\sS_{(0,17)}=\{(x,y)\in \R_+^2\mid y< 17x\}$$ 
is a strict homogeneous mean. 
\item 
\label{item:ineq}
The functions $\mu_1$ and $\mu_2$ on $\sS_{(-1,\infty)}$ satisfy inequalities 
\begin{align}
\label{eq:mu1>mu2}
&\mu_1(x,y)- \mu_2(x,y)
\ge 0,\\
\label{eq:mu1+mu2>0}
&\mu_1(x,y)+\mu_2(x,y)
> 0,
\end{align}
for any $(x,y)\in \sS_{(-1,\infty)}$.  
If $x\ne y$ then \eqref{eq:mu1>mu2} holds strictly. \\
If $(x,y)\in \sS_{(-1,1]}$ then 
\begin{equation}\mu_1(x,y)+\mu_2(x,y)
\ge x+y\ (>0).
\end{equation}
\end{enumerate}
\end{proposition}
\begin{proof}
\itref{item:expr}
By straightforward calculations, we have
\begin{align*}
&m_A(x,\mu_0(x,y))=\frac{x+\sqrt{x\cdot (x+y)/2}}{2}
=\frac{2x+\sqrt{2x(x+y)}}{4}=\mu_1(x,y),\\
&\nu(x,\mu_0(x,y))=2m_H(x,\mu_0(x,y))-m_A(x,\mu_0(x,y))\\
=&
\frac{4x\sqrt{x\cdot(x+y)/2}}{x+\sqrt{x\cdot(x+y)/2}}
-\frac{2x+\sqrt{2x(x+y)}}{4}
=
\frac{16x\sqrt{2x(x+y)}-(2x+\sqrt{2x(x+y)})^2}
{4(2x+\sqrt{2x(x+y)})}\\
=&\frac{-3x^2+6x\sqrt{2x(x+y)}-xy}
{2(2x+\sqrt{2x(x+y)})}
=\mu_2(x,y).
\end{align*}
We can extend them to functions on $\sS_{(-1,\infty)}$ since 
$\mu_0(x,y)$ is defined on it, and takes positive values.

\medskip\noindent\itref{item:means}
Since $m_A$ and $\mu_0$ are strict homogeneous means on 
$\R^2$ and $\sS_{(-1,\infty)}$, and $pr_x:(x,y)\to x$ 
is a homogeneous mean on $\R^2$, 
$\mu_1(x,y)=m_A(pr_x(x,y),\mu_0(x,y))$ is a strict homogeneous mean 
on $\sS_{(-1,\infty)}=\R^2 \cap \sS_{(-1,\infty)}$ by Lemma \ref{lem:comp}. 

Under the assumption $(x,y)\in \sS_{(-1,1]}$, i.e., $x>0$ and $-x<y\le x$,  
we show the inequality \eqref{eq:increase-mu2}. Since 
\begin{align*}
\mu_2(x,y)-y=&
2[m_H(x,\mu_0)-m_A(x,\mu_0)]+m_A(x,\mu_0)-y,\\
&2[m_A(x,\mu_0)-m_H(x,\mu_0)]
=\frac{(x-\mu_0)^2}{x+\mu_0},
\end{align*}
we have 
$$\mu_2(x,y)-y\ge 0\quad \Leftrightarrow \quad
m_A(x,\mu_0)-y\ge \frac{(x-\mu_0)^2}{x+\mu_0},
$$
where $\mu_0$ denotes $\mu_0(x,y)$.
Note that 
$$y<m_A(x,y)<\mu_0<m_A(x,\mu_0)<x,\quad 0\le \frac{x-\mu_0}{x+\mu_0}<1.$$
Then we have 
$$
m_A(x,\mu_0)-y\ge 
m_A(x,y)-y=x-m_A(x,y)\ge 
x-\mu_0
\ge \frac{x-\mu_0}{x+\mu_0}(x-\mu_0)
=\frac{(x-\mu_0)^2}{x+\mu_0}.
$$

To show the restriction of $\mu_2$ to $\sS_{(0,17)}$ is a strict 
homogeneous mean,  
we check that $\sS_{(1/3,3)}$ contains the image of $\sS_{(0,17)}$ under the map 
$$(x,y)\mapsto (x,\mu_0(x,y)).$$
In fact, if $0<y<x$ then $\mu_0(x,y)$ satisfies inequalities 
$$\frac{x}{3}<\frac{x}{\sqrt{2}}<\mu_0(x,y)<x$$
by \eqref{eq:ineq1}. If $x\le y<17x$ then 
$$x\le \mu_0(x,y)=m_G(x,m_A(x,y))<m_G(x,m_A(x,17x))=m_G(x,9x)=3x.$$
Thus we have $(x,\mu_0(x,y))\in \sS_{(1/3,3)}$ for any $(x,y)\in \sS_{(0,17)}$.
Since $\nu$ is a strict homogeneous mean on $\sS_{(1/3,3)}$ as shown 
in Lemma \ref{lem:nu}, 
and the restriction of $\mu_0$ to $\sS_{(0,17)}$ is strict homogeneous
and that of $\pr_x$ to $\sS_{(0,17)}$ is homogeneous, 
$\mu_2(x,y)=\nu(pr_x(x,y),\mu_0(x,y))$ is a strict homogeneous mean 
on $\sS_{(0,17)}$ by Lemma \ref{lem:comp}. 

\medskip\noindent\itref{item:ineq}
We have 
\begin{align*}\mu_1(x,y)- \mu_2(x,y)&
=2(m_A(x,\mu_0(x,y))-m_H(x,\mu_0(x,y)))
\ge0,\\
\mu_1(x,y)+\mu_2(x,y)&
=2m_H(x,\mu_0(x,y))>0,
\end{align*}
by using \eqref{eq:compound-ex-mu1,2} together with 
the positivity of $\mu_0(x,y)$ on $\sS_{(-1,\infty)}$. 
Since 
$$m_A(x,\mu_0(x,y))=m_H(x,\mu_0(x,y))\Leftrightarrow 
x=\mu_0(x,y) \Leftrightarrow x=y,$$  if $x\ne y$ then 
\eqref{eq:mu1>mu2} holds strictly.

If $(x,y)\in \sS_{(-1,1]}$ then $x>0$ and $-x<y\le x$, we have  
$$\mu_1(x,y)+\mu_2(x,y)
=2m_H(x,\mu_0(x,y))\ge 2m_H(m_A(x,y),m_A(x,y))=x+y\ (>0)
$$
since $x$ and $\mu_0(x,y)$ are greater than or equal to $m_A(x,y)$. 
\end{proof}

\begin{remark}
If $(x,y)\in \sS_{(-7/9,17)}$ then $(x,\mu_0(x,y)) \in \sS_{(1/3,3)}$.
\end{remark}


By using functions $\mu_1$ and $\mu_2$ on $\sS_{(-1,\infty)}$,  
we define a pair of sequences $\{x_n\}$ and $\{y_n\}$ 
for $(x,y)\in \sS_{(-1,\infty)}$ as 

\begin{equation}
\label{eq:sequences}
(x_0,y_0)=(x,y),\quad 
(x_{n+1},y_{n+1})=(\mu_1(x_{n},y_{n}),\mu_2(x_{n},y_{n})). 
\end{equation}

\begin{lemma}
\label{lem:limit}
The sequences $\{x_n\}$ and $\{y_n\}$
converge as $n\to \infty$, and satisfy   
$$\lim_{n\to \infty} x_n=\lim_{n\to \infty} y_n>0.$$
\end{lemma}
\begin{proof}
For any $(x,y)\in \sS_{(-1,\infty)}$, 
$x_1=\mu_1(x,y)$ and $y_1=\mu_2(x,y)$  
satisfy $x_1>0$, $-x_1<y_1\le x_1$ by 
\eqref{eq:mu1>mu2} and \eqref{eq:mu1+mu2>0} in Proposition \ref{prop:mu1mu2}. 
Since $(x_1,y_1)\in \sS_{(-1,1]}$, 
and $\mu_1$ is a strict mean on $\sS_{(-1,\infty)}$, 
$x_2=\mu_1(x_1,y_1)$ satisfies 
$$0<x_2,\quad y_1\le x_2\le x_1.$$
Since $(x_1,y_1)\in \sS_{(-1,1]}$ and    
\eqref{eq:increase-mu2}, \eqref{eq:mu1>mu2}, \eqref{eq:mu1+mu2>0} in 
Proposition \ref{prop:mu1mu2},  
$y_2=\mu_2(x_1,y_1)$  
satisfies 
$$y_1\le y_2,\quad -x_2<y_2\le x_2.$$
Thus we have 
$$ (x_2,y_2)\in \sS_{(-1,1]},\quad y_1\le y_2\le x_2\le x_1.$$
Inductively, we have 
$$(x_n,y_n)\in \sS_{(-1,1]},\quad y_1\le \cdots \le y_n\le x_n\le \cdots\le x_1$$
for any $n\in \N$ and any $(x,y)\in \sS_{(-1,\infty)}$. 

Since the sequences $\{x_n\}$ and $\{y_n\}$ $(n\ge 1)$ 
are monotonous and bounded, 
they converge. We set 
$$\lim_{n\to \infty} x_n=\xi ,\quad \lim_{n\to \infty} y_n=\eta,
$$
which satisfy $\xi \ge \eta$. 
Since $x_n\ge y_n$ and $m_A(x_n,y_n)\le \mu_0(x_n,y_n)\le x_n$, we have
$$
m_A(x_{n+1},y_{n+1})=m_H(x_n,\mu_0(x_n,y_n))\ge m_H(m_A(x_n,y_n),m_A(x_n,y_n))= 
m_A(x_n,y_n)>0
$$
which yields $\xi +\eta>0$.
By considering the limit $n\to \infty$ for 
$$x_{n+1}=\mu_1(x_n,y_n),$$
we have 
$$\xi =\frac{\xi +\mu_0(\xi ,\eta)}{2} \quad \Leftrightarrow \quad \xi =\mu_0(\xi ,\eta)
\quad \Leftrightarrow \quad \xi ^2=\xi \cdot \frac{\xi +\eta}{2}. 
$$
The last equality is equivalent to $\xi =\eta$ or $\xi =0$. 
If $\xi =0$ then $0=\xi \ge \eta$, which contradicts $\xi +\eta>0$. 
Hence we have $\xi =\eta>0$. 
\end{proof}

\begin{theorem}
\label{th:lim-HGS-rep}
For the sequences $\{x_n\}$ and $\{y_n\}$ given in \eqref{eq:sequences}, 
there exists $N\in \N_0$ such that 
$0<y_n\le x_n$ for any $n\ge N$, 
and that  
$$\lim_{n\to \infty} x_n=\lim_{n\to \infty} y_n=
\frac{x_N}{F(\frac{1}{4},\frac{1}{4},1;1-\frac{y_N^2}{x_N^2})}. 
$$
\end{theorem}

\begin{proof}
Since the sequences $\{x_n\}$ and $\{y_n\}$ satisfy $y_n\le x_n$ for 
$n\ge 1$, and converge to a positive real number
by Lemma \ref{lem:limit}, the number $N\in \N_0$ exists.  
We have 
\begin{align*}
\frac{x_N} {F(\frac{1}{4},\frac{1}{4},1;1-\frac{y_N^2}{x_N^2})}
&=\frac{m_1(x_N,y_N)}{F(\frac{1}{4},\frac{1}{4},1;1-\frac{m_2(x_N,y_N)^2}
{m_1(x_N,y_N)^2})}=\frac{x_{N+1}} 
{F(\frac{1}{4},\frac{1}{4},1;1-\frac{y_{N+1}^2}{x_{N+1}^2})}\\
&=\frac{x_{N+2}} 
{F(\frac{1}{4},\frac{1}{4},1;1-\frac{y_{N+2}^2}{x_{N+2}^2})}=
\cdots =\frac{x_{N+k}} 
{F(\frac{1}{4},\frac{1}{4},1;1-\frac{y_{N+k}^2}{x_{N+k}^2})}
\end{align*}
by Theorem \ref{th:trans}, our definition of $\mu_1,\mu_2$ 
in \eqref{eq:mu1,2} and $0<\frac{y_n}{x_n}\le \frac{y_{n+1}}{x_{n+1}}\le 1$
for $n\ge N$. Let $k$ tend to $\infty$ in the last term of 
the previous equality, then 
$$
\frac{x_N}{F(\frac{1}{4},\frac{1}{4},1;1-\frac{y_N^2}{x_N^2})}
= \lim_{n\to \infty}x_n,
$$
holds, since $\ds{\lim_{n\to\infty} x_n=\lim_{n\to\infty} y_n}$ and 
$\lim\limits_{z\to 1}F(\frac{1}{4},\frac{1}{4},1;1-z^2)=1$.  
 \end{proof}

\begin{proof}[Alternative Proof]
For the real numbers $x_N,y_N$, 
there exists a pure imaginary number $\tau\in \D_{12}$ such that 
$$1-\frac{y_N^2}{x_N^2}=\zeta(\tau)=\frac{4\h_{01}(\tau)^4\h_{10}(\tau)^4}
{\h_{00}(\tau)^8}=1-\Big(\frac{2\h_{01}(\tau)^4-\h_{00}(\tau)^4}
{\h_{00}(\tau)^4}\Big)^2
$$
by Fact \ref{fact:Inv-S-map} (\ref{item:zeta}). Thus we have a positive 
real number $\xi$ such that 
$$x_N=\xi\cdot \h_{00}(\tau)^2,\quad 
y_N=\xi\cdot \frac{2\h_{01}(\tau)^4-\h_{00}(\tau)^4}{\h_{00}(\tau)^2}.
$$
Since 
\begin{align*}
\mu_1(x_N,y_N)&=
\xi\cdot \mu_1\big(\h_{00}(\tau)^2,
\frac{2\h_{01}(\tau)^4-\h_{00}(\tau)^4}{\h_{00}(\tau)^2}\big)
=\xi\cdot \h_{00}(2\tau)^2,
\\
\mu_2(x_N,y_N)&=
\xi\cdot \mu_2\big(\h_{00}(\tau)^2,
\frac{2\h_{01}(\tau)^4-\h_{00}(\tau)^4}{\h_{00}(\tau)^2}\big)
=\xi\cdot \frac{2\h_{01}(2\tau)^4-\h_{00}(2\tau)^4}{\h_{00}(2\tau)^2}
\end{align*}
by our definition of $\mu_1$ and $\mu_2$,  $x_n$ and $y_n$ 
admit expressions 
$$x_{N+k}=\xi\cdot \h_{00}(2^k\tau)^2,\quad 
y_{N+k}=\xi\cdot 
\frac{2\h_{01}(2^k\tau)^4-\h_{00}(2^k\tau)^4}{\h_{00}(2^k\tau)^2}.$$
Hence we have 
$$
\lim_{n\to\infty} x_n=\xi\cdot \lim_{k\to \infty} \h_{00}(2^k\tau)^2
=\xi=\frac{x_N}{\h_{00}(\tau)^2}
=\frac{x_N}{F(\frac{1}{4},\frac{1}{4},1;1-\frac{y_N^2}{x_N^2})}
$$
by Theorem \ref{th:theta=HGS}. 
\end{proof}

\begin{remark}
For $(x,y)\in \sS_{(-1,\infty)}$ with $y<0$, 
we have $(x,-y)\in \sS_{(-1,\infty)}$ and 
$$F(\frac{1}{4},\frac{1}{4},1;1-\frac{y^2}{x^2})=
F(\frac{1}{4},\frac{1}{4},1;1-\frac{(-y)^2}{x^2}).
$$
However, the pair of sequences in \eqref{eq:sequences} with 
$(x_0,y_0)=(x,y)$ is different from that with $(x_0,y_0)=(x,-y)$, 
and their limits as $n\to \infty$ do not coincide in general.
\end{remark}


\end{document}